\documentclass{article}

\usepackage{color, amsmath, amssymb, amsthm, url, todonotes}
\usepackage[a4paper,body={16cm,23cm}]{geometry}
\usepackage[all]{xy}

\theoremstyle{remark}
\newtheorem*{defn}{Definition}
\newtheorem*{gwqn}{Grunwald--Wang problem for elliptic curves}

\definecolor{dblue}{rgb}{0,0,0.7}
\newtheoremstyle{mythm}{11pt}{11pt}{\it\color{dblue}}{}{\bf\color{dblue}}{.}{ }{}
\theoremstyle{mythm}
\newtheorem{thm}{Theorem}
\newtheorem{prop}[thm]{Proposition}
\newtheorem{lem}[thm]{Lemma}
\newtheorem{cor}[thm]{Corollary}

\DeclareMathOperator{\GL}{GL}
\DeclareMathOperator{\SL}{SL}
\DeclareMathOperator{\PGL}{PGL}
\DeclareMathOperator{\Gal}{Gal}
\DeclareMathOperator{\Hom}{Hom}
\DeclareMathOperator{\Sel}{Sel}
\DeclareMathOperator{\im}{im}
\DeclareMathOperator{\res}{res}
\DeclareMathOperator{\Aut}{Aut}
\DeclareMathOperator{\Mat}{Mat}
\DeclareMathOperator{\Fr}{Fr}

\DeclareFontFamily{U}{wncy}{}
\DeclareFontShape{U}{wncy}{m}{n}{<->wncyr10}{}
\DeclareSymbolFont{mcy}{U}{wncy}{m}{n}
\DeclareMathSymbol{\Sha}{\mathord}{mcy}{"58}

\renewcommand{\geq}{\geqslant}
\renewcommand{\leq}{\leqslant}

\newcommand{\QQ}{\mathbb{Q}}
\newcommand{\ZZ}{\mathbb{Z}}
\newcommand{\FF}{\mathbb{F}}
\newcommand{\PP}{\mathbb{P}}
\newcommand{\sm}[4]{\bigl(\begin{smallmatrix}#1 & #2 \\ #3 & #4 \end{smallmatrix}\bigr)}

\begin{document}
 \title{Vanishing of some Galois cohomology groups for elliptic curves}
 \author{Tyler Lawson\thanks{Partially supported by NSF DMS-1206008.} \and Christian Wuthrich}
 \maketitle

 \subsection*{Abstract}
 Let $E/\QQ$ be an elliptic curve and $p$ be a prime number, and let $G$ be the Galois group of the extension of $\QQ$ obtained by adjoining the coordinates of the $p$-torsion points on $E$. We determine all cases when the Galois cohomology group $H^1\bigl( G, E[p]\bigr)$ does not vanish, and investigate the analogous question for $E[p^i]$ when $i>1$. We include an application to the verification of certain cases of the Birch and Swinnerton-Dyer conjecture, and another application to the Grunwald--Wang problem for elliptic curves.

 \section{Introduction}
 Let $E$ be an elliptic curve over $\QQ$ and $p$ a prime number. Denote by $K$ the Galois extension of $\QQ$ obtained by adjoining the coordinates of the $p$-torsion points on $E$ and let $G$ be the Galois group of $K/\QQ$. The Galois action on the $p$-torsion points $E[p]$ identifies $G$ with a subgroup of $\GL\bigl(E[p]\bigr) \cong \GL_2(\FF_p)$ via the representation $\rho\colon \Gal\bigl(\bar \QQ/\QQ\bigr)\to \GL\bigl(E[p]\bigr)$. A celebrated theorem of Serre~\cite{serregl2} shows that $G$ is equal to the full group $\GL_2(\FF_p)$ for all but finitely many primes $p$ when the curve is fixed.

 We are interested in the vanishing of the Galois cohomology group $H^1\bigl(G, E[p]\bigr)$; see~\cite{serre_coh_gal} or~\cite{cnf} for the basic definitions of Galois cohomology. This specific cohomology group appears as an obstruction in various contexts. For instance, Kolyvagin's work uses the vanishing of this group in the case $G$ is equal to $\GL_2(\FF_p)$ (see Proposition~9.1 in~\cite{gross}).  The following first theorem characterizes completely when this cohomology group does not vanish, answering a question at \cite{mo_question}.

 \begin{thm}\label{thm1}
   Fix a prime $p$. Let $E/\QQ$ be an elliptic curve, $K = \QQ(E[p])$, and $G$ the Galois group of $K/\QQ$. Then $H^1\bigl(G, E[p]\bigr)$ is trivial except in the following cases:
  \begin{itemize}
   \item $p=3$, there is a rational point of order $3$ on $E$, and there are no other isogenies of degree $3$ from $E$ that are defined over $\QQ$.
   \item $p=5$ and the quadratic twist of $E$ by $D=5$ has a rational point of order $5$, but no other isogenies of degree $5$ defined over $\QQ$.
   \item $p=11$ and $E$ is the curve labeled as 121c2 in Cremona's tables~\cite{cremona}, given by the global minimal equation
 $y^2 + x\,y = x^3 + x^2 - 3632\,x + 82757$.
  \end{itemize}
In each of these cases, $H^1\bigl(G, E[p]\bigr)$ has $p$ elements.
 \end{thm}

 Partial results on this question have appeared in various sources. For instance, Lemma~10 in~\cite{coates} by Coates shows that $H^1\bigl(G, E[p]\bigr)$ vanishes when $E[p]$ is irreducible as a Galois module. Section~3 in~\cite{ciperiani_stix} also treats related questions.

 The above result extends to elliptic curves $E$ over more general number fields $F$ if we assume that $F\cap \QQ(\mu_p) = \QQ$, where $\QQ(\mu_p)$ is the field generated by $p$-th roots of unity. Rather than a single elliptic curve for $p>5$, one finds possibly infinitely many exceptions for $p=11$ and $p=17$, but only finitely many further exceptions for each $p>17$ and none for all $p$ such that $p\equiv 1 \pmod{3}$. See Theorem~\ref{thm1_general} for a precise statement.

 Next, we address the analogous question for $E[p^i]$ for $i>1$, but assuming that $p>3$.

 \begin{thm}\label{thmi}
  Fix a prime $p>3$. Let $E/\QQ$ be an elliptic curve, $K_i = \QQ(E[p^i])$ the extension of $\QQ$ obtained by adjoining the coordinates of all $p^i$-torsion points, and $G_i$ the Galois group of $K_i/\QQ$. Then $H^1\bigl(G_2,E[p^2]\bigr)$ is trivial if and only if $H^1\bigl(G_i,E[p^i]\bigr)$ is trivial for all $i\geq 2$. This vanishing holds if and only if $(E,p)$ is not among the following cases:
  \begin{itemize}
    \item $p=5$ or $p=7$ and $E$ contains a rational $p$-torsion point.
    \item $p=5$ and there is an isogeny $\varphi\colon E\to E'$ of degree $5$ defined over $\QQ$ and the quadratic twist by $D=5$ of $E$ contains a rational $5$-torsion point.
    \item $p=5$ and there is an isogeny $\varphi\colon E\to E'$ of degree $5$ defined over $\QQ$ but none of degree $25$ and the quadratic twist by $D=5$ of $E'$ contains a rational $5$-torsion point.
    \item $p=5$ and $E$ admits an isogeny $E\to E'\to E''$ of degree $25$ defined over $\QQ$ and $E'$ contains a rational $5$-torsion point.
    \item $p=11$ and $E$ is 121c1 or 121c2.
  \end{itemize}
 \end{thm}

 Again, we will also obtain some results that are valid over more general base fields $F$ with $F\cap\QQ(\mu_p) = \QQ$, and some that are valid for $p = 3$. See Section~\ref{results_i_sec}.

 This more general question has also been investigated before, and Cha has obtained results in this direction in~\cite{cha}. He proved the vanishing of $H^1\bigl(G_i,E[p^i]\bigr)$ when $p>3$, the curve has semi-stable reduction at an unramified place above $p$, and $E$ does not have a rational $p$-torsion point. He also describes when this cohomology group vanishes for $p=3$ under his assumptions. The method of proof is similar.

 The results in Theorem~\ref{thmi} can be applied to the Grunwald--Wang problem for elliptic curves as formulated by Dvornicich and Zannier in~\cite{dvornicich_zannier}. In Proposition~\ref{div9_prop}, we give an example of an elliptic curve $E/\QQ$ with a point $P\in E(\QQ)$ divisible by $m=9$ in $E(\QQ_\ell)$ for almost all primes $\ell$ but not divisible by $9$ in $E(\QQ)$. Previously, the only known examples~\cite{dvornicich_zannier_4} were with $m=4$. In Theorem~\ref{div_thm}, we also give a simplified proof of the result in~\cite{paladino_5} that it is impossible to find such a point $P$ when $m=p^2$ and $p>3$.

 The paper is structured as follows. We begin with some background in Section~\ref{sec:prel-notat}, both establishing notation and reducing to cases where the Galois group $G$ does not contain a nontrivial homothety. In Section~\ref{sec:main-theorem} we prove a general form of Theorem~\ref{thm1}. Section~\ref{sec:second-cohom} establishes a vanishing result for $H^2$. In Section~\ref{sec:appl-bsd} we give an application to verifying cases of the Birch and Swinnerton-Dyer conjecture, correcting an oversight in \cite{steinetal}. Our main results classifying the vanishing of $H^1(G_i, E[p^i])$ are then discussed in Section~\ref{results_i_sec}, and some supplementary numerical computations for $H^1(G_2, E[p^2])$ are included in Section~\ref{sec:numer-comp}. Finally, in Section~\ref{sec:divisibility} we give the application to the Grunwald--Wang problem for elliptic curves.

\subsection*{Acknowledgments}
 It is our pleasure to thank Jean Gillibert and John Coates for interesting comments and suggestions. We are also grateful to Brendan Creutz for pointing us to~\cite{creutz}.

 \section{Preliminaries and notation}\label{sec:prel-notat}
 Throughout this paper $E$ will be an elliptic curve defined over a number field $F$ and $p$ will be a prime number. We will denote by $K=F\bigl(E[p]\bigr)$ the number field obtained by adjoining the coordinates of the $p$-torsion points to $F$. Let $G$ be the Galois group of $K/F$. More generally, for $i\geq 1$ we let $K_i = F\bigl(E[p^i]\bigr)$ and $G_i=\Gal(K_i/F)$. The faithful actions of $G_i$ on $E[p^i]$ give embeddings $G_i \hookrightarrow \Aut(E[p^i]) \cong \GL_2(\ZZ/p^i)$, and so we may regard them as subgroups.

 We will also use the groups $H_i = \Gal(K_{i+1}/K_i)$ and $M_i = \Gal(K_i/K)$. We note that, as $H_i$ is the kernel of the map $G_{i+1} \to G_i$, it is identified with a subgroup of
\[
\ker\Bigl(\GL_2(\ZZ/p^{i+1}) \to \GL_2(\ZZ/p^i)\Bigr) \cong \Mat_2(\FF_p),
\]
where the conjugation action of $G_{i+1} \subset \GL_2(\ZZ/p^{i+1})$ is by the adjoint representation. Therefore, all elements in $H_i$ have order $p$ and commute with the elements of $M_{i+1}$ inside $G_{i+1}$.

 In summary, we have the following situation:
 
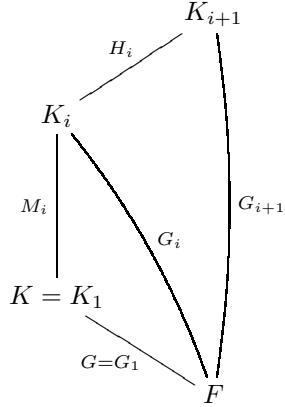
\begin{equation*}
  \xymatrix{
  & K_{i+1} \ar@{-}[dl]_{H_i} \ar@{-}@/^/[dddd]^{G_{i+1}} \\
  K_i \ar@{-}[dd]_{M_i} \ar@{-}@/^/[dddr]^{G_i}& \\
  & \\
  K=K_1 \ar@{-}[dr]_{G=G_1}& \\
  & F
  }
 \end{equation*}
 We will later use the inflation-restriction sequence
 \begin{equation}\label{infres_seq}
  \xymatrix@1{
    0\ar[r] & H^1\bigl( G_i, E[p^j]\bigr) \ar[r]^-{\inf} &
              H^1\bigl( G_{i+1}, E[p^j]\bigr) \ar[r]^-{\res} &
              H^1\bigl( H_i, E[p^j]\bigr)^{G_i} \ar[r] & H^2\bigl(G_i, E[p^j] \bigr)
  }
 \end{equation}
 which is valid for all $1\leq j\leq i$. In inductive arguments, we will also use that the short exact sequence
\[
\xymatrix@1{0\ar[r]&E[p]\ar[r]&E[p^j]\ar[r]& E[p^{j-1}]\ar[r]&0}
\]
gives a long exact sequence
  \begin{equation}\label{mult_p_seq}
  \xymatrix@1{E(F)[p^{j-1}]\ar[r]& H^1\bigl(G_i, E[p]\bigr)\ar[r] & H^1\bigl(G_i,E[p^j]\bigr)\ar[r]& H^1\bigl(G_i, E[p^{j-1}]\bigr).}
  \end{equation}

 As mentioned in the introduction, these cohomology groups only start to be interesting when $E[p]$ is reducible. The following argument for this is given in~\cite{cha} as Theorem~7.
 \begin{lem}\label{homothety_lem}
   If $G$ contains a non-trivial homothety, then $H^1\bigl(G_i, E[p^i]\bigr)=0$.
 \end{lem}
 \begin{proof}
  Let $g$ be a non-trivial homothety. Since $g$ is central, $\langle g\rangle$ is a normal subgroup in $G$. Consider the inflation-restriction sequence
  \begin{equation*}
   \xymatrix@1{ 0\ar[r] & H^1\bigl( G/\langle g\rangle, E[p]^{g=1} \bigr) \ar[r] &
                          H^1\bigl( G, E[p]\bigr) \ar[r] &
                          H^1\bigl( \langle g \rangle, E[p])       }
  \end{equation*}
  The homothety $g$ cannot have fixed points in $E[p]$; in particular $E(F)[p]=0$.
  The left-hand side cohomology group in the above sequence is therefore trivial. The right-hand side is also trivial because $\langle g\rangle $ is of order coprime to $p$.

  We assume by induction that $H^1\bigl(G_{i}, E[p^{i}]\bigr)$ and $H^1\bigl(G_{i}, E[p]\bigr)$ are both trivial. By assumption, the restriction maps
  \begin{equation*}
    \xymatrix@R-4ex{
       H^1\bigl(G_{i+1}, E[p^i]\bigr) \ar[r] & H^1\bigl(H_{i}, E[p^i]\bigr)^{G_i} \cong \Hom\bigl(H_{i}, E[p^i]\bigr)^{G_i} \\
      H^1\bigl(G_{i+1}, E[p]\bigr) \ar[r] & H^1\bigl(H_{i}, E[p]\bigr)^{G_i} \cong \Hom\bigl(H_{i}, E[p]\bigr)^{G_i}
     }
  \end{equation*}
  from~\eqref{infres_seq} are both injective. Note that the target groups are actually equal because all elements in $H_i$ have order $p$. Since $M_{i+1}$ and $H_i$ commute, the action of $G_i$ on $H_i$ factors through $G$, so the target in both cases is $\Hom\bigl(H_i,E[p]\bigr)^{G}$.

 The homothety $g$ acts trivially on $H_i$ and non-trivially on any non-zero point in $E[p]$. Therefore, there are no homomorphisms from $H_i$ to $E[p]$ which are fixed by $g$. It follows that $H^1\bigl(G_{i+1}, E[p^i]\bigr) $ and $H^1\bigl(G_{i+1}, E[p]\bigr) $ are both trivial. The exact sequence~\eqref{mult_p_seq} now implies that $H^1\bigl(G_{i+1}, E[p^{i+1}]\bigr)$ is also trivial.
 \end{proof}

 \begin{lem}\label{homothety_borel_lem}
   Suppose $p>2$. Assume that $G$ does not contain a non-trivial homothety and $F \cap \QQ(\mu_p) = \QQ$. Then $G$ is contained in a Borel subgroup.
 \end{lem}
 \begin{proof}
   By the Weil pairing, the determinant of $\rho$ is the Teichm\"uller character $\omega$ describing the action of Galois on the $p$-th roots of unity $\mu_p$. The assumption $F \cap \QQ(\mu_p) = \QQ$ implies that $\det \colon G\to \FF_p^{\times}$ must be surjective.

   Assume first that $p>3$. We fix a basis of $E[p]$ and view $G$ as a subgroup of $\GL_2(\FF_p)$. By the classification of maximal subgroups of $\GL_2(\FF_p)$, we have to show that the following cases can not occur: $G$ is a subgroup of the normalizer of a split Cartan group, $G$ is a subgroup of the normalizer of a non-split Cartan group, or $G$ maps to an exceptional group $A_4$, $A_5$ or $S_4$ in $\PGL_2(\FF_p)$.

   Suppose $G$ is a subgroup of the group of diagonal and anti-diagonal matrices, which is the normalizer of a split Cartan subgroup. Suppose moreover that $G$ is not a subgroup of the diagonal matrices. The square of $\sm{0}{b}{c}{0}\in G$ is the homothety by $bc$. Therefore, all anti-diagonal elements in $G$ must be of the form $\sm{0}{c^{-1}}{c}{0}$. Multiplying this with a diagonal element $\sm{u}{0}{0}{v}$ in $G$ then shows that all diagonal elements must have determinant $1$. Hence the determinant would not be surjective for $p>3$.

   Next, suppose that $G$ is a subgroup of the normalizer of a non-split Cartan group. Since $G$ contains no non-trivial homothety, the image of $G$ in $\PGL_2(\FF_p)$ is isomorphic to $G$. In other words, $G$ must be a subgroup of a dihedral group of order $2(p+1)$. No such group could have a surjective map onto $\FF_p^{\times}$ if $p>3$.

   Finally, assume that $G$ is exceptional. As before, our hypothesis implies that $G$ is isomorphic to a subgroup of $A_4$, $A_5$ or $S_4$. However the only case in which we could have a surjective map onto $\FF_p^{\times}$ with $p>3$ is when $p=5$ and $G$ is a cyclic group of order $4$ in $S_4$. However, as $\FF_5$ contains the fourth roots of unity $\mu_4$, all such subgroups are diagonalizable in $\GL_2(\FF_5)$.

   We now return to the case $p=3$. By assumption, $G$ is isomorphic to its image in $\PGL_2(\FF_p)$, which is the full symmetric group on the four elements $\PP^1(\FF_3)$. Since the determinant is surjective, the image of $G$ cannot be contained in the alternating group. Therefore it is not transitive on $\PP^1(\FF_3)$, and $G$ is contained in a Borel subgroup.
\end{proof}

 From now on we will suppose that $\varphi\colon E\to E'$ is an isogeny of degree $p$ defined over $F$, and write $E[\varphi]$ for its kernel. The dual isogeny is denoted by $\hat\varphi\colon E'\to E$. We will now also fix a basis of $E[p]$ with the property that the first point belongs to $E[\varphi]$. In this basis, the Galois representation $\rho \colon \Gal\bigl(\bar{F}/F\bigr) \to \GL\bigl(E[p]\bigr) \cong \GL_2(\FF_p)$ now takes values in the Borel subgroup of upper triangular matrices. We will write $\chi\colon \Gal\bigl(\bar F/F\bigr)\to \FF_p^{\times}$ for the character of the Galois group on $E'[\hat\varphi]$. Then the character on $E[\varphi]$ is $\omega\chi^{-1}$, where $\omega$ is the Teichm\"uller character introduced above. The representation now is of the form $\rho =\sm{\omega\chi^{-1}}{*}{0}{\chi}$.

 \begin{cor}\label{one_isogeny_cor}
  Suppose $p>2$. If $F \cap \QQ(\mu_p) = \QQ$ and the group $H^1\bigl(G, E[p]\bigr)$ is non-trivial, then $E$ admits exactly one isogeny $\varphi\colon E \to E'$ of degree $p$ that is defined over $F$.
 \end{cor}
 \begin{proof}
  By Lemma~\ref{homothety_lem}, we know that there is no non-trivial homothety in $G$. Then Lemma~\ref{homothety_borel_lem} implies that $G$ is contained in a Borel subgroup. Hence there is a subgroup of order $p$ in $E[p]$ fixed by the Galois group. If there were a second subgroup of order $p$ fixed by the Galois group, then in a suitable basis of $E[p]$ the group $G$ would consist of diagonal matrices. It would follow that $G$ has order coprime to $p$ and therefore that the cohomology group is trivial. Therefore, there is a unique isogeny defined over $F$ of degree $p$.
 \end{proof}

 \section{Proof of Theorem~\ref{thm1}}\label{sec:main-theorem}

  We begin by assuming that $E$ is defined over a number field $F$ such that $F\cap \QQ(\mu_p) = \QQ$.

  \begin{lem}
    The cohomology group $H^1(G; E[2])$ always vanishes.
  \end{lem}
  \begin{proof}
  The group $\GL_2(\FF_2)$ is isomorphic to the symmetric group on $3$ letters. For any cyclic subgroup of $\GL_2(\FF_2)$ of order $2$ generated by $h$, we may compute $H^1\bigl(\langle h\rangle,E[2]\bigr)$ as the quotient of the kernel of the norm $N_G = 1+h$ on $E[2]$ modulo the image of $h-1$. Because $p=2$, this group is trivial.

  For a general subgroup $G \leq \GL_2(\FF_2)$, let $H$ be the intersection of $G$ with the normal subgroup of order $3$. We have $H^1\bigl(H, E[2]\bigr)=0$ because the order of $H$ is coprime to $2$. We also have $H^1\bigl(G/H, E[2]^H\bigr)=0$ because $H$ is either of order $3$ and only fixes $0$ in $E[2]$, or $H$ is trivial and this group is $H^1\bigl(\langle h\rangle, E[2]\bigr) =0$. By the inflation-restriction sequence, we conclude that $H^1\bigl(G,E[2]\bigr)=0$.
  \end{proof}

\begin{lem}
\label{normalizer_action_lem}
  Let $H < \GL_2(\FF_p)$ be the subgroup generated by $h = \sm{1}{1}{0}{1}$. We have an isomorphism
  \begin{equation*}
    H^1\bigl(H,E[p]\bigr) \cong \FF_p,
  \end{equation*}
  and the action of an element $g = \sm{u}{w}{0}{v}$ in the normalizer $N(H)$ of $H$ on this cohomology group is multiplication by $u^{-1} v^2$.
\end{lem}

\begin{proof}
The cohomology of the cyclic group $H$ is computed to be
 \begin{equation*}
   H^1\bigl(H,E[p]\bigr) \cong \frac{ \ker\bigl( \sum_{a=0}^{p-1} h^a \bigr)}{\im(h-1)} = \frac{\ker{(0)}}{\im \sm{0}{1}{0}{0} } \cong \FF_p.
 \end{equation*}
 The explicit isomorphism $i\colon H^1\bigl(H,E[p]\bigr)\to \FF_p$ sends a cocycle $\xi\colon H\to E[p]$ to the second coordinate of $\xi(h)$.
 Let now $g= \sm{u}{w}{0}{v}$ be an element of $N(H)$ with $u, v\in\FF_p^{\times}$. Then the action of $g$ on $\xi \in H^1\bigl(H,E[p]\bigr)$ is as follows.
 \begin{align*}
   \bigl( g \star \xi \bigr)(h) &= g\, \xi\bigl( g^{-1} h g \bigr) \\
                            &= g\, \xi \bigl( h^{u^{-1}v} \bigr) \\
                            &= g \, \bigl( h^{u^{-1}v - 1 } + \cdots + h+ 1\bigr) \xi(h)\\
                            &= \sm{u}{0}{0}{v} \,\sm{u^{-1}v}{ * }{ 0 }{ u^{-1}v }\, \bigl(\begin{smallmatrix} * \\ i(\xi) \end{smallmatrix}\bigr)
 \end{align*}
 Here the terms denoted by $*$ are unknown entries which do not alter the result that
 \begin{equation*}
   i\bigl( g \star \xi \bigr) = u^{-1} \, v^2\, i(\xi).\qedhere
 \end{equation*}
\end{proof}

 For the remainder of this section we will assume that $p>2$ and $E$ satisfies $H^1\bigl(G,E[p]\bigr)\neq 0$; we wish to show that we fall into one of the cases listed in the Theorem~\ref{thm1}.

\begin{lem}
  Suppose $p>2$. Then $G$ satisfies $H^1\bigl(G,E[p]\bigr)\neq 0$ if and only if $p \not\equiv 1 \pmod{3}$ and there exists a basis of $E[p]$ such that $G$ consists of all matrices of the form $\sm{v^2}{w}{0}{v}$ with $v\in \FF_p^{\times}$ and $w\in \FF_p$. In this case, the cohomology group is isomorphic to $\FF_p$ and the representation $\rho$ is of the form $\sm{\chi^2}{*}{0}{\chi}$, where $\chi^3$ is the Teichm\"uller character $\omega$.
\end{lem}

\begin{proof}
 By Corollary~\ref{one_isogeny_cor}, we may view $G$ as a group of upper triangular matrices containing the subgroup $H$ generated by element $h = \sm{1}{1}{0}{1}$ of order $p$.

Since $H$ is a normal subgroup of $G$, we can use the inflation-restriction sequence to show that
 \begin{equation*}
   \xymatrix@1{H^1\bigl(G,E[p]\bigr) \ar[r] & H^1\bigl(H,E[p]\bigr)^{G/H} }
 \end{equation*}
 is an isomorphism because $G/H$ is of order coprime to $p$. Because we assumed that $H^1\bigl(G,E[p]\bigr)$ is non-trivial, by Lemma~\ref{normalizer_action_lem} we must have that $G/H$ acts trivially on $H^1\bigl(H,E[p]\bigr)$, that $H^1\bigl(G,E[p]\bigr)$ has precisely $p$ elements, and that all elements in $G$ must be of the form $\sm{v^2}{w}{0}{v}$ with $w\in \FF_p$ and $v\in\FF_p^{\times}$.

Recall that the character $\chi$ is such that $\rho = \sm{\omega\chi^{-1}}{*}{0}{\chi}$. We now deduce that $\chi^2 = \omega\chi^{-1}$ and hence $\chi^3 = \omega$. Since we assumed $F\cap \QQ(\mu_p) = \QQ$, the determinant $\omega$ from $G$ to $\FF_p^{\times}$ must be surjective. As the determinant of the typical element in $G$ is $v^3$ with $v\in \FF_p^{\times}$, we must conclude that either $p=3$ or $p\equiv 2 \pmod{3}$, and that $G$ is equal to the group of all matrices of the form  $\sm{v^2}{w}{0}{v}$.
\end{proof}

\begin{cor}
  If $p=3$, we have $H^1\bigl(G,E[3]\bigr)\neq 0$ if and only if $E$ has a $3$-torsion point and no other isogenies defined over $F$.
\end{cor}

\begin{proof}
  This can only occur if the group $G$ is the group of matrices of the form $\sm{1}{w}{0}{v}$ of order $6$. This is precisely the case when $E(F)[3]$ is of order $3$ and no other isogenies are defined over $F$.
\end{proof}

\begin{lem}
  If $p=5$, we have $H^1\bigl(G,E[5]\bigr)\neq 0$ if and only if the quadratic twist of $E$ by $D = 5$ has a $5$-torsion point and no other isogenies defined over $F$.
\end{lem}

\begin{proof}
  This happens precisely when we have
  \begin{equation*}
   \rho = \begin{pmatrix} \omega^2 & * \\ 0 & \omega^{-1} \end{pmatrix}
  \end{equation*}
 Here $\omega^2$ is the quadratic character corresponding to the non-trivial extension $F\bigl(\sqrt{5}\bigr)/F$ contained in $F(\mu_5)$. Let $E^{\dagger}$ be the quadratic twist of $E$ by $D=5$. Then we have the desired form of representation $\rho$ if and only if the representation $\rho^{\dagger}$ on $E^{\dagger}[5]$ is now of the form $\sm{1}{*}{0}{\omega}$. We conclude that this occurs if and only if $E^{\dagger}(F)[5]$ has five points and $E^{\dagger}$ has no other isogenies of degree $5$ defined over $F$.
\end{proof}

\begin{thm}\label{thm1_general}
  Let $E$ be an elliptic curve defined over a number field $F$ with $F\cap \QQ(\mu_p) = \QQ$. Let $K = F\bigl(E[p]\bigr)$ and $G=\Gal(F/K)$. Then $H^1\bigl(G,E[p]\bigr)=0$ except in the following cases:
  \begin{itemize}
    \item $p=3$, there is a rational $3$-torsion point in $E(F)$, and there are no other $3$-isogenies from $E$ defined over $F$.
    \item $p=5$ and the quadratic twist of $E$ by $D=5$ has a rational point of order $5$, but no other isogenies of degree $5$ defined over $F$.
    \item $p\geq 11$, $p\equiv 2 \pmod{3}$, there is a unique isogeny $\varphi\colon E \to E'$ of degree $p$ defined over $F$, its kernel $E[\varphi]$ acquires a rational point over $F\cdot \QQ(\mu_p)^{+}$, and $E[\varphi] \cong \mu_p^{\otimes (p+1)/3}$.
  \end{itemize}
  There are only finitely many cases for each prime $p$ with $p>17$.
\end{thm}

\begin{proof}
 The only remaining cases to prove are those where $p > 5$. As we may assume $p\equiv 2 \pmod{3}$, one sees that
 \begin{equation*}
   \rho = \begin{pmatrix} \omega^{\frac{p+1}{3} } & * \\ 0 & \omega^{\frac{2-p}{3}} \end{pmatrix} .
 \end{equation*}
 This explains the condition in the cases $p\geq 11$ in the above list.

 The curve $E$ and its unique isogeny $\varphi$ of degree $p$ defined over $F$ represent a point on the modular curve $Y_0(p)$ defined over $F$. For $p=11$ and $p=17$, the curve $Y_0(p)$ is of genus $1$; for all larger primes $p\equiv 2\pmod{3}$ it is of genus at least two. Therefore there are only finitely many $\bar\QQ$-isomorphism classes of curves $E/F$ with an isogeny of degree $p$ defined over $F$. Only a single twist in each class can have $\rho$ of the above shape. Hence there are only finitely many exceptions for $p>17$.
\end{proof}

We specialize now to the field $F=\QQ$ where the points on $Y_0(p)$ are well-known.

\begin{lem}\label{121c2_lem}
  If $F = \QQ$ and $p > 5$, we have $H^1\bigl(G,E[p]\bigr)\neq 0$ if and only if $E$ is the curve labeled as 121c2 in Cremona's tables.
\end{lem}

\begin{proof}
 For all those $p$, there are only a finite number of $\bar \QQ$-isomorphism classes of elliptic curves $E$ with a $p$-isogeny defined over $\QQ$. Mazur's theorem~\cite{mazur_isogenies} shows that there are no rational points on $Y_0(p)$ except for three points on $Y_0(11)$ and two points on $Y_0(17)$. All of these five examples have no other automorphisms than $\pm 1$. Hence, all elliptic curves $E/\QQ$ representing one of them are quadratic twists of each other.

 Let us first look at $p=11$. The $j$-invariants of the three families are $-121$, $-32768$, and $-24729001$, and the representation $\rho$ must now be of the form $\sm{\omega^4}{*}{0}{\omega^7}$. We start with the last. The curve 121c2 is an example of an elliptic curve with $j$-invariant $-24729001$. Using SageMath~\cite{sage}, we find a point $P$ of order $11$ in $E\bigl(\QQ(\mu_{11})\bigr)$. Its $x$-coordinate in the global minimal model given above is $11\zeta^9 + 11\zeta^8+22\zeta^7+22\zeta^6+22\zeta^5+22\zeta^4+11\zeta^3+11\zeta^2+39$, where $\zeta$ is a primitive $11$-th root of unity. One finds that $\sigma(P) = 5P$ for the Galois element with $\sigma(\zeta) = \zeta^2$. Therefore the action of Galois on the group generated by $P$ is given by $\omega^4$. The isogeny with $P$ in its kernel is defined over $\QQ$ and it is the only isogeny on $E$ defined over $\QQ$. Therefore the group $G$ is precisely of the form required. Hence $H^1\bigl(G, E[p]\bigr)$ has $p$ elements. No quadratic twist of $E$ could have the same property.

 With similar computation one finds that the group $G$ for the curve 121b1 with $j$-invariant $-32768$ is of the form $\sm{\omega^8}{*}{0}{\omega^3}$ and for the curve 121c1 with $j$-invariant $-121$ it is $\sm{\omega^7}{*}{0}{\omega^4}$. No quadratic twist of these curves could have the required form for $G$.

 For $p=17$, the representation $\rho$ must now be of the form $\sm{\omega^6}{*}{0}{\omega^{11}}$. In particular, for any prime $\ell\neq 11$ of good reduction for $E$, the Frobenius element is sent to a matrix of the form $\sm{\ell^6}{*}{0}{\ell^{11}}$. We conclude that we must have $\ell^6 + \ell^{11} \equiv a_{\ell} \pmod{ p}$, where $a_{\ell}$ is the trace of Frobenius. This gives an easy criterion to rule out specific curves.

 There are two $j$-invariants of elliptic curves that admit a $17$-isogeny over $\QQ$: $-297756989/2$ and  $-882216989/131072$. In fact, these values were computed by V\'elu and published on page 80 of~\cite{mfov4}. We pick a curve $E$ for each of these $j$-invariants. The curves 14450p1 and 14450n1 are examples. Now for both curves, it is easy to show that $\pm 3^6 \pm 3^{11} \not\equiv a_3 \pmod{17}$ for any choice of the signs as $a_3=\pm 2$. Therefore no quadratic twist of $E$ will satisfy the congruence that we need. Thus $H^1\bigl( G, E[p]\bigr) = 0$ for all curves with a degree-$17$ isogeny. Similar computations were done by Greenberg in Remark~2.1.2 in~\cite{greenberg}.

\end{proof}

This concludes the proof of Theorem~\ref{thm1}.

\section{Vanishing of the second cohomology}\label{sec:second-cohom}

  We continue to assume that $E$ is defined over a number field $F$ such that $F\cap \QQ(\mu_p) = \QQ$.

\begin{lem}\label{h2_lem}
  Let $p$ be a prime. Then $H^2\bigl(G, E[p]\bigr) =0$ except if $p>2$, $E$ admits a $p$-isogeny $\varphi\colon E \to E'$ and no other $p$-isogenies over $F$, and $E'[\hat\varphi]$ contains an $F$-rational $p$-torsion point. If this cohomology group is non-zero then it contains $p$ elements.
\end{lem}

We could also write the condition in the lemma as either that $E[\varphi] \cong \mu_p$ or that $\chi$ is trivial.

\begin{proof}
  As before, only the cases when $p$ divides the order of $G$ are of interest.

  We again discuss the case $p=2$ separately. A Sylow subgroup of $G$ is a cyclic group of order $2$ generated by $h$, and the restriction $H^2\bigl(G,E[p]\bigr) \to H^2\bigl(\langle h\rangle,E[p]\bigr)$ is an inclusion. However, $H^2\bigl(\langle h\rangle,E[p]\bigr)$ can be computed as the Tate cohomology group $\hat{H}^0\bigl(\langle h\rangle,E[p]\bigr)$, which is zero.

  For $p > 2$, we have to deal with the cases when $G$ contains $\SL\bigl(E[p]\bigr)$ and when $G$ is contained in a Borel subgroup.

  In the first case, $G$ is actually the full group $\GL\bigl(E[p]\bigr)$ as the Weil pairing forces the determinant to be surjective. If $Z$ is the center of $G$, then $H^i\bigl(Z,E[p]\bigr)=0$ for all $i\geq 0$. The Hochschild-Serre spectral sequence implies that $H^i\bigl(G,E[p]\bigr)=0$ for all $i \geq 0$.

  Now we may assume that $G$ is contained in the Borel subgroup of upper-triangular matrices.  If there is more that one isomorphism class of $p$-isogeny leaving $E$ which is defined over $F$, then $G$ is of order coprime to $p$ and hence $ H^2\bigl(G, E[p]\bigr) =0$. Therefore, we may assume that $G$ contains the unique $p$-Sylow $H$ generated by $h=\sm{1}{1}{0}{1}$. Since $H$ is normal and $G/H$ is of order coprime to $p$, the restriction
  \begin{equation*}
   H^2\bigl(G,E[p]\bigr) \cong H^2\bigl(H,E[p]\bigr)^{G/H}
  \end{equation*}
  is an isomorphism.

  Fix an injective homomorphism $\psi\colon H\to \QQ/\ZZ$. Let $\delta\colon H^1\bigl(H,\QQ/\ZZ\bigr) \to H^2\bigl(H,\ZZ\bigr)$ be the connecting homomorphism. Then we have an isomorphism $\hat{H}^0\bigl( H, E[p]\bigr) \to H^2\bigl(H,E[p]\bigr)$ given by sending a point $P\in E[p]$ to the cup product $\delta\psi\cup P$. For $p>2$, the Tate cohomology group $\hat{H}^0\bigl( H, E[p]\bigr)$ is equal to the usual cohomology group $H^0\bigl(H,E[p]\bigr) = E[\varphi]$, which has $p$ elements.

  Let $g=\sm{u}{w}{0}{v}\in G$. On the one hand, it acts on $P$ by multiplication by $u$. On the other hand, it acts on $\psi$ by multiplication by $u^{-1} v$ because
  \begin{equation*}
   (g \star \psi) (h) = g\, \psi\bigl(g^{-1}hg\bigr) = \psi\bigl(h^{u^{-1}v}\bigr) = u^{-1}v \,\psi(h).
  \end{equation*}
  It follows that $g$ acts on the generator of $H^2\bigl(H,E[p]\bigr)$ by multiplication by $uu^{-1}v = v$. Unless all such $g \in G$ have $v=1$, we conclude that the second cohomology group vanishes. Otherwise it has $p$ elements, and this occurs if and only if $E'[\hat\varphi]$ contains a rational $p$-torsion point.
\end{proof}

\section{Application to the conjecture of Birch and Swinnerton-Dyer and $p$-descent}\label{sec:appl-bsd}

The vanishing of the Galois cohomology group we consider is used when trying to extend Kolyvagin's results to find a sharper bound on the Birch and Swinnerton-Dyer conjecture for elliptic curves of analytic rank at most $1$. This was the original motivation in Cha's work~\cite{cha}. In~\cite{steinetal}, the authors attempt to extend Cha's results, but there is a mistake in the proof of their Lemma~5.4 and consequently their Theorem~3.5 is not correct. The latter is also copied as Theorem~5.3 in~\cite{miller}. Using our results above, we can now state and prove a corrected version of Theorem~3.5 in~\cite{steinetal}. We refer to the original paper for the notations.

\begin{thm}\label{heegner_thm}
  Let $E/\QQ$ be an elliptic curve of analytic rank at most $1$. Let $p$ be an odd prime. Let $F$ a quadratic imaginary field satisfying the Heegner hypothesis and suppose $p$ does not ramify in $F/\QQ$. Suppose that $(E,p)$ does not appear in the list of Theorem~\ref{thm1} and that $E$ is not isogenous to an elliptic curve over $\QQ$ such that the dual isogeny contains a rational $p$-torsion point. Then the $p$-adic valuation of the order of the Tate-Shafarevich group is bounded by twice the index of the Heegner point.
\end{thm}
\begin{proof}
  In their proof, only the vanishing of $H^1\bigl(G,E[p]\bigr)$ and $H^2\bigl(G,E[p]\bigr)$ are needed for the argument. Under our assumptions they both vanish by Theorem~\ref{thm1} and Lemma~\ref{h2_lem}. One has also to note that, as pointed out in~\cite{miller}, the assumption in their theorem that E does not admit complex multiplication is not used in the proof. Finally, the paper~\cite{steinetal} needs that $F$ is not included in $K=\QQ\bigl(E[p]\bigr)$ to conclude that $H^i\bigl(\Gal(F(E[p])/F),E[p]\bigr)$ also vanishes for $i=1$ and $2$. This is guaranteed by the Heegner hypothesis and the assumption that $p$ does not ramify in $F$, as $F$ and $K$ then have disjoint sets of ramified primes.
\end{proof}

The following is a short-cut in the usual $p$-descent for $E=$121c2 and $p=11$. It is not a new result as it appears already in~\cite{miller_stoll} as Example~7.4. However it illustrates that the non-trivial class in $H^1\bigl(G,E[p]\bigr)$ can be of use.
\begin{prop}\label{11descent_prop}
  The Tate-Shafarevich group of the curve 121c2 does not contain any non-trivial elements of order $11$. The full Birch and Swinnerton-Dyer conjecture holds for this curve.
\end{prop}
\begin{proof}
 Set $p=11$. Let $\varphi\colon E\to E'$ be the $p$-isogeny defined over $\QQ$. We saw before that $E[\varphi] \cong \FF_p(4)$ and $E'[\hat\varphi]\cong\FF_p(7)$ where $\FF_p(k)$ is the $1$-dimensional $\FF_p$-vector space with the Galois group acting by the character $\omega^k$.

 Let $F$ be the maximal extension of $\QQ$ which is unramified at all finite places $\ell\neq p$. Write $\mathcal{G} = \Gal\bigl(F/\QQ\bigl)$ and $\mathcal{H} = \Gal\bigl(F/\QQ(\zeta)\bigr)$ where $\zeta$ is a primitive $p$-th root of $1$. Let $\Gamma=\mathcal{G}/\mathcal{H}$. Since $\vert\Gamma\vert$ is coprime to $p$, we have an isomorphism $H^1\bigl(\mathcal{G},E[\varphi]\bigr) \cong H^1(\mathcal{H},E[\varphi]\bigr)^{\Gamma}$. Now Dirichlet's unit theorem can be used to compute
 \begin{equation*}
  H^1\bigl(\mathcal{H}, \FF_p(1)\bigr) =  H^1\bigl(\mathcal{H}, \mu_p\bigr)\cong \FF_p(1) \oplus \bigoplus_{i=0}^{4} \FF_p(2\,i)
 \end{equation*}
 as a $\FF_p[\Gamma]$-module; see for instance Corollary~8.6.12 (or 8.7.3 in the second edition) in~\cite{cnf}. Since $H^1\bigl(\mathcal{H},\FF_p(k)\bigr)\cong H^1\bigl(\mathcal{H},\FF_p(1)\bigl)(k-1)$, the group $H^1\bigl(\mathcal{G},\FF_p(k)\bigr)$ is a sum of copies of $\FF_p$ corresponding to the copies of $\FF_p(1-k)$ in $H^1\bigl(\mathcal{H}, \FF_p(1)\bigr)$. We deduce that $H^1\bigl(\mathcal{G},E[\varphi]\bigr)$ is trivial and that $H^1\bigl(\mathcal{G},E'[\hat\varphi]\bigr)$ is $1$-dimensional.

 Since $K/\QQ$ is only ramified at $p$, we have an inflation map $H^1\bigl(G,E'[\hat\varphi]\bigr)\to H^1\bigl(\mathcal{G}, E'[\hat\varphi]\bigl)$. By Theorem~\ref{thm1} and the above, this is now an isomorphism and our explicit cocycle $\xi$ can be viewed as a generator for $H^1\bigl(\mathcal{G}, E'[\hat\varphi]\bigl)$.

 The $\hat\varphi$-Selmer group $\Sel^{\hat\varphi}$ is defined to be the kernel of the map
 \begin{equation*}
  H^1\bigl(\mathcal{G},E'[\hat\varphi]\bigr) \to H^1\bigl(\QQ_p,E'\bigr)[\hat\varphi].
 \end{equation*}
 An explicit local computation shows that $\hat\varphi\colon E'(\QQ_{p})\to E(\QQ_p)$ is surjective.  Therefore $H^1\bigl(\QQ_p,E'\bigr)[\hat\varphi] \cong  H^1\bigl(\QQ_p,E'[\hat\varphi]\bigr)$. Since $K/\QQ$ is totally ramified at $p$, the decomposition group of $K/\QQ$ at the unique place above $p$ in $K$ is equal to $G$. Therefore $\xi$ also inflates to a non-trivial element in $H^1\bigl(\QQ_p,E'[\hat\varphi]\bigr)$. It follows that the generator of $H^1\bigl(\mathcal{G},E'[\hat\varphi]\bigr)$ does not lie in the Selmer group. Therefore  $\Sel^{\hat\varphi}$ is trivial.

 Since $ H^1\bigl(\mathcal{G},E[\varphi]\bigr)=0$, the $\varphi$-Selmer group $\Sel^{\varphi}$ is trivial. The usual exact sequence
 \begin{equation*}
  \xymatrix@1{ \Sel^{\varphi} \ar[r] & \Sel^{p}(E/\QQ) \ar[r] & \Sel^{\hat\varphi} }
 \end{equation*}
 shows now that the $p$-Selmer group $\Sel^{p}(E/\QQ)$ is trivial. Therefore the rank of $E$ is zero and the $p$-primary part of the Tate-Shafarevich group $\Sha(E/\QQ)$ is trivial.

As explained in Theorem~8.5 in~\cite{miller}, the only prime at which one has to check the Birch and Swinnerton-Dyer conjecture after the Heegner point computations done there is $p=11$. Therefore, this completes the proof of the conjecture for this specific elliptic curve.
\end{proof}

The main result of~\cite{miller} (based on~\cite{miller_stoll} and~\cite{creutz_miller}) by Miller and his collaborators states that the Birch and Swinnerton-Dyer conjecture holds for all elliptic curves of conductor at most $5000$ and analytic rank at most $1$. As a consequence of the error in~\cite{steinetal}, the verification for some curves in this list is not complete. The following is a description how we performed the necessary computations to fill in the gaps for all these curves. See also~\cite{sage_bug} for the correction of the corresponding bug in SageMath.

From the change in Theorem~\ref{heegner_thm}, it follows that only curves $E$ contained in the list of Theorem~\ref{thm1} could have been affected when verifying the $p$-part of the conjecture.
 The case 121c2 was verified in Proposition~\ref{11descent_prop}. The exceptional cases with $p=3$ are already dealt with in Theorem~9.1 in~\cite{miller_stoll}, as they were already considered exceptional cases there. That only leaves the curves with non-vanishing $H^1\bigl(G,E[5]\bigr)$. For the following list of curves, we had to perform a $5$-descent to verify the conjecture: 50a3, 50a4, 75a2, 150b3, 150b4, 175c2, 275b1, 325d2, 550b1, 550f3, 775c1, 950a1, 1050d2, 1425b1, 1450a1, 1650b1, 1650b2, 1650c2, 1650d2, 1950b2, 1975d1, 2175f2, 2350e2, 2550f2, 2850a1, 2850a2, 2850g2, 2950a1, 3075d1, 3075g2, 3325c1, 3550d1, 3850k2, 3950a1, 4350a1, 4350a2, 4425c1, 4450a1, 4450f2, 4650e1, 4650k2, 4650m2. The methods in~\cite{miller_stoll} are sufficient in all these cases. If the rank is $1$, then even the weaker bound in their Corollary~7.3 is enough. Otherwise, if the rank is $0$, the Selmer groups for $\varphi$ and $\hat\varphi$ are trivial as one finds quickly by looking at a few local conditions.

\section{Results for $i>1$}\label{results_i_sec}

We now turn to the question of finding all cases of elliptic curves $E/F$ and primes $p$ such that the group $H^1\bigl(G_i,E[p^i]\bigr)$ does not vanish for some $i>1$. We continue to assume that $F\cap\QQ(\mu_p) = \QQ$ and we will assume now that $p>2$.

By Lemma~\ref{homothety_lem} and Lemma~\ref{homothety_borel_lem}, we know that all these groups vanish unless there is an isogeny $\varphi\colon E\to E'$ defined over $F$. Therefore, we may continue to assume the existence of $\varphi$ and that the group $G$ is contained in the Borel subgroup of upper triangular matrices. This fixes (up to scalar) the first basis element of $E[p]$ and we still have some flexibility about the second; if there is a second subgroup of $E[p]$ fixed by the Galois group, we will choose the second basis element in there. Unlike in the case $i=1$, we may not yet assume that $p$ divides the order of $G$.

In what follows we will write expressions like $G=\sm{1}{*}{0}{*}$. By this we mean that $G$ is equal to the group of all matrices of this form in $\GL_2(\FF_p)$, so $*$ on the diagonal can take any non-zero value and $*$ in the top right corner can be any value in $\FF_p$.

Let $M$ be the additive group of $2\times 2$-matrices with coefficients in $\FF_p$. Then $G\leq \GL_2(\FF_p)$ acts on $M$ by conjugation. We would like to determine $\Hom_G\bigl(M, E[p]\bigr)$. We do so by computing first $\Hom_G\bigl(M,E[\varphi]\bigr)$ and $\Hom_G\bigl(M,E'[\hat\varphi]\bigr)$.

\begin{lem}\label{dim_hom_phi_lem}
 Suppose first $p>3$. The group $\Hom_G\bigl(M,E[\varphi]\bigr)$ is trivial except in the following cases.
 \begin{itemize}
 \item If $G=\sm{1}{0}{0}{*}$ in a suitable basis of $E[p]$, then $\Hom_G\bigl(M,E[\varphi]\bigr)$ has dimension $2$ over $\FF_p$.
 \item If $G =\sm{*}{0}{0}{1}$ in a suitable basis of $E[p]$, this group has dimension $1$.
 \item If $G=\sm{1}{*}{0}{*}$, this group has dimension $1$.
 \item If $G \leq \bigl\{ \sm{u}{w}{0}{u^2} \bigm\vert u \in \FF_p^\times, w \in \FF_p\bigr\}$, this group has dimension $1$.
 \end{itemize}
 If $p=3$, the list is the same with one modification to the second and to the last case above.
 \begin{itemize}
  \item If $G =\sm{*}{0}{0}{1}$ in a suitable basis of $E[3]$, this group has dimension $2$.
 \end{itemize}
\end{lem}

\begin{proof}
 If $f\colon M \to E[\varphi]$ is fixed by $g\in G$, then $f(m)=g\cdot f(g^{-1}mg)$ for all $m\in M$. Let $\alpha$, $\beta$, $\gamma$, and $\delta$ be the images in $E[\varphi]$ under $f$ of $\sm{1}{0}{0}{0}$, $\sm{0}{1}{0}{0}$, $\sm{0}{0}{1}{0}$, and $\sm{0}{0}{0}{1}$ respectively. Then the above equation for $m$ being one of the these four matrices yields four equations that have to hold for all $g=\sm{u}{w}{0}{v}\in G$:
 \begin{equation}\label{hom_phi_eq}\begin{aligned}
   \alpha &= u \cdot (\alpha + u^{-1}w\,\beta)\\
   \beta  &= u\cdot (u^{-1}v\,\beta )\\
   \gamma &= u\cdot ( -v^{-1}w\,\alpha-u^{-1}v^{-1}w^2\,\beta + uv^{-1}\,\gamma +v^{-1}w\,\delta)\\
   \delta &= u\cdot (\delta -u^{-1}w\,\beta)
   \end{aligned}
 \end{equation}
 From these equations, we deduce the following:
 \begin{align*}
   f \text{ is fixed by }\sm{1}{1}{0}{1} &\iff  \beta=0\text{ and } \alpha = \delta\\
   f \text{ is fixed by }\sm{1}{0}{0}{v} \text{ for some $v\neq 1$} &\iff \beta=\gamma=0 \\
   f \text{ is fixed by }\sm{u}{0}{0}{1} \text{ for some $u\neq \pm1$} &\iff \alpha=\gamma=\delta=0\\
   f \text{ is fixed by }\sm{-1}{0}{0}{1} &\iff \alpha=\delta=0
 \end{align*}
 Assume first that $G$ is contained in $\sm{1}{*}{0}{*}$. Since the determinant must be surjective, $G$ is either $\sm{1}{*}{0}{*}$ or $\sm{1}{0}{0}{*}$, after choosing a suitable second basis element for $E[p]$. In both cases, the above allows us to verify the statements in the lemma. The case when $G$ is contained in $\sm{*}{*}{0}{1}$ is very similar, except that when $p=3$, in which case we are in the group of matrices with $v=u^2$ and we can only apply the fourth equation instead of the third.

 Assume now that $G$ contains an element $\sm{u}{w}{0}{v}$ with $v\neq 1$ and one with $u\neq 1$. Then $\beta=0$ by the second equation in~\eqref{hom_phi_eq}. From the last two equations, we deduce that $\alpha=\delta=0$. Now the equations~\eqref{hom_phi_eq} simplify to one equation $(1-u^2v^{-1})\gamma = 0$. Therefore, if $G$ is contained in the group of matrices with $v=u^2$, then the dimension of $\Hom_G\bigl(M,E[\varphi]\bigr)$ is $1$ and $p>3$ as otherwise all $g\in G$ have $v=1$, otherwise the space is trivial.
\end{proof}

Recall that $E'[\hat\varphi]$ is the kernel of the dual isogeny.
\begin{lem}
Suppose $p > 2$. The group $\Hom_G\bigl(M,E'[\hat\varphi]\bigr)$ is trivial except in the following cases. If $G$ is contained in $\sm{1}{*}{0}{*}$ or if $G=\sm{*}{*}{0}{1}$ or if $G = \bigl\{ \sm{v^2}{0}{0}{v}\bigm\vert v\in\FF_p^{\times}\bigr\}$ in a suitable basis for $E[p]$, then $\Hom_G\bigl(M,E'[\hat\varphi]\bigr)$ has dimension $1$. If $G=\sm{*}{0}{0}{1}$ in a suitable basis of $E[p]$, then it has dimension $2$.
\end{lem}

\begin{proof}
 This is analogous to the proof of the previous lemma. The equations~\eqref{hom_phi_eq} become equations where the $u$ at the start of the right hand side of each equation is replaced by a $v$. This new set of equations can be rewritten as follows.
  \begin{equation}\label{hom_phihat_eq}\begin{aligned}
   (1-v)\alpha &= u^{-1}vw\,\beta\\
   (1-u^{-1}v^2)\beta  &= 0\\
   (1-u)\gamma &= w(\delta-\alpha) - u^{-1}w^2 \beta\\
  (1-v)\delta &=  -u^{-1}vw\,\beta
   \end{aligned}
 \end{equation}
 From here, the computations are again straightforward for the cases when $G$ is contained in $\sm{*}{*}{0}{1}$ or $\sm{1}{*}{0}{*}$. If $G$ is contained in the group $\bigl\{\sm{v^2}{w}{0}{v}\bigm\vert v \in \FF_p^\times, w\in \FF_p\bigr\}$, it is either equal to this group, in which case the cohomology group in  question is trivial, or it is equal to a subgroup of order $p-1$. In the latter case, we may change the choice of basis of $E[p]$ to get $G$ to be equal to
 $\bigl\{ \sm{v^2}{0}{0}{v}\bigm\vert v\in\FF_p^{\times}\bigr\}$, in which case $\alpha=\gamma=\delta=0$, but $\beta$ is free. In all other cases it is trivial.
\end{proof}

The exact sequence
\begin{equation}\label{hom_G_seq}
\xymatrix@1{0\ar[r]& \Hom_G\bigl(M,E[\varphi]\bigr)\ar[r]&\Hom_G\bigl(M,E[p]\bigr)\ar[r]^-{\varphi} & \Hom_G\bigl(M,E'[\hat\varphi]\bigr)}
\end{equation}
connects the results from the previous two lemmas.
\begin{prop}\label{hom_prop}
 If $p > 3$, the group $\Hom_G\bigl(M,E[p]\bigr)$ vanishes except when, for some choice of basis of $E[p]$, it is one of the following subgroups.
  \begin{center}\begin{tabular}{|c|c|c|c|c|c|}
    \hline &&&&&\\[-2ex]
    $G$                                       & $=\sm{1}{0}{0}{*}$ & $=\sm{*}{0}{0}{1}$ & $=\sm{1}{*}{0}{*}$ &  $\leq\sm{u}{*}{0}{u^2}$ & $=\sm{v^2}{0}{0}{v}$ \\[3pt] \hline &&&&& \\[-2ex]
    $\dim_{\FF_p}\Hom_G\bigl( M , E[p] \bigr)$ & $3$                & $3$                & $2$                & $1$
                 & $1$ \\ \hline
  \end{tabular}\end{center}
 If $p = 3$, the group $\Hom_G\bigl(M,E[p]\bigr)$ vanishes except when, for some choice of basis of $E[p]$, it is one of the following subgroups.
  \begin{center}\begin{tabular}{|c|c|c|c|c|}
    \hline &&&&\\[-2ex]
    $G$                                       & $=\sm{1}{0}{0}{*}$ & $=\sm{*}{0}{0}{1}$ & $=\sm{1}{*}{0}{*}$ &  $=\sm{*}{*}{0}{1}$ \\[3pt] \hline &&&& \\[-2ex]
    $\dim_{\FF_3}\Hom_G\bigl( M , E[3] \bigr)$ & $3$                & $4$                & $2$                & $1$
                \\ \hline
  \end{tabular}\end{center}
\end{prop}
Here we have chosen a suitable second basis element in $E[p]$ as in the previous lemmas. Of course, the first two cases are in fact the same when the basis elements are swapped.
\begin{proof}
  If $\Hom_G\bigl(M,E'[\hat\varphi]\bigr)=0$, then the exact sequence~\eqref{hom_G_seq} reduces this to Lemma~\ref{dim_hom_phi_lem}. Otherwise, we have to check if the homomorphisms $f\colon M \to E'[\hat\varphi]$ lift to homomorphisms $e\colon M\to E[p]$ that are $G$-equivariant. In the following four cases, they all lift indeed. We will just give the explicit map which form a basis of $\Hom_G\bigl(M,E[p]\bigr)$ modulo the image from $\Hom_G\bigl(M,E[\varphi]\bigr)$. One can verify without difficult that they are $G$-equivariant.
  \begin{center}\begin{tabular}{|c|c|c|c|c|}\hline &&&&\\[-2ex]
    $G$                & $=\sm{1}{0}{0}{*}$ & $=\sm{*}{0}{0}{1}$ & $=\sm{1}{*}{0}{*}$  & $=\sm{v^2}{0}{0}{v}$ \\[3pt] \hline &&&& \\[-2ex]
    $e\sm{a}{b}{c}{d}$ & $\bigl(\begin{smallmatrix} a \\ c \end{smallmatrix}\bigr)$ & $\bigl(\begin{smallmatrix} 0 \\ a \end{smallmatrix}\bigr)$ and $\bigl(\begin{smallmatrix} 0 \\ d \end{smallmatrix}\bigr)$ & $\bigl(\begin{smallmatrix} a \\ c \end{smallmatrix}\bigr)$  & $\bigl(\begin{smallmatrix} 0 \\ b \end{smallmatrix}\bigr)$
    \\ \hline
  \end{tabular}\end{center}
  There is only the case $G=\sm{*}{*}{0}{1}$ left to treat. The generator of $\Hom_G\bigl(M,E'[\hat\varphi]\bigr)$ is given by $f\sm{a}{b}{c}{d} = a+d$. We will show that $f$ does not lift to a map $e\colon M\to E[p]$. Denote by $\bigl(\begin{smallmatrix} \alpha \\ 1 \end{smallmatrix}\bigr)$ the image of $\sm{1}{0}{0}{0}$ under such an $e$ and by $\bigl(\begin{smallmatrix} \beta \\ 0 \end{smallmatrix}\bigr)$ the image of $\sm{0}{1}{0}{0}$. Then we must have for all $u\neq 1$ and $w$ in $\FF_p$ that
  \begin{equation*}
     \bigl(\begin{smallmatrix} \beta \\ 0 \end{smallmatrix}\bigr) = \sm{u}{w}{0}{1} e \sm{0}{u^{-1}w}{0}{0} = \sm{u}{w}{0}{1} \bigl(\begin{smallmatrix} u^{-1}w\beta \\ 0 \end{smallmatrix}\bigr) =  \bigl(\begin{smallmatrix} w\beta \\ 0 \end{smallmatrix}\bigr).
  \end{equation*}
  Hence $\beta=0$. Again for all $u$ and $w$, we should have that
  \begin{equation*}
    \bigl(\begin{smallmatrix} \alpha \\ 1 \end{smallmatrix}\bigr) = \sm{u}{w}{0}{1} e\sm{1}{u^{-1}w}{0}{0} = \sm{u}{w}{0}{1} \bigl(\begin{smallmatrix} \alpha \\ 1 \end{smallmatrix}\bigr) =  \bigl(\begin{smallmatrix} u \alpha + w \\ 1 \end{smallmatrix}\bigr).
  \end{equation*}
  However, this cannot hold for all choices no matter what $\alpha$ is.
\end{proof}

\begin{defn}
  Let $E/F$ be an elliptic curve. We will say that $G_i$ is \emph{greatest possible} if it consists of all the matrices in $\GL_2(\ZZ/p^i\ZZ)$ that reduce to a matrix in $G$ modulo $p$. Equivalently, $M_i$ is the kernel of the map $\GL_2(\ZZ/p^i) \to \GL_2(\FF_p)$.
\end{defn}

We will show that if $p > 2$ and $i > 1$, then $G_i$ is greatest possible if and only if $G_2$ is greatest possible: Since $G_i \to G_{i-1}$ is surjective, by induction it suffices to prove that the kernel $H_{i-1}$ contains all matrices of the form $1 + p^{i-1} A \in \GL_2(\ZZ/p^i)$. If $G_2$ is greatest possible, then for any $A \in M_2(\FF_p)$ there exists an element $g \in G_i$ whose image in $\GL_2(\ZZ/p^2)$ is $1 + pA$. Then $g^{p^{i-2}}$ has image $(1 + p^{i-1} A)$ in $\GL_2(\ZZ/p^i)$ by taking binomial expansions, and so $G_i$ contains all of $H_{i-1}$.

\begin{prop}\label{hi_vanishes_prop}
 Let $p>2$ be a prime and let $E/F$ be an elliptic curve. Suppose $G$ lies in the Borel subgroup of upper triangular matrices and that $G_2$ is greatest possible. If $G$ is not among the exceptional cases in Theorem~\ref{thm1} or in Proposition~\ref{hom_prop}, then
 $H^1\bigl(G_i,E[p^j]\bigr) =0$ for all $i\geq j \geq 1$.
\end{prop}
\begin{proof}
  The short exact sequence~\eqref{mult_p_seq} implies that if $H^1\bigl(G_{i+1},E[p^{j-1}]\bigr)$ and $H^1\bigl(G_{i+1},E[p]\bigr)$ are zero, then so is $H^1\bigl(G_{i+1},E[p^j]\bigr)$. By induction on $j$, it suffices to prove the proposition in the case $j=1$.

  For $i=1$, the statement follows from Theorem~\ref{thm1}. We assume now that it holds for $i\geq 1$. By assumption $M_{i+1}$ is isomorphic to the group $(1 + p \Mat_2\bigl(\ZZ/p^i\bigr)) \subset \GL_2(\ZZ/p^{i+1})$ and $H_i$ is isomorphic to the group $M$ of all matrices with coefficients in $\FF_p$.
  Using Proposition~\ref{hom_prop}, we find
  \begin{equation*}
   H^1\bigl(M_{i+1},E[p]\bigr)^G = \Hom_G\bigl(M_{i+1},E[p]\bigr)\cong \Hom_G\bigl(M,E[p]\bigr) =0
  \end{equation*}
  because all elements in the kernel of the map $M_{i+1} \to M$ are $p$'th powers. (Note that this requires $p > 2$.) Now considering the inflation-restriction sequence
  \begin{equation}\label{inf_res_M_seq}
    \xymatrix@1{ 0\ar[r] & H^1\bigl(G,E[p]\bigr)\ar[r] &
  H^1\bigl(G_{i+1},E[p]\bigr)\ar[r]& H^1\bigl(M_{i+1},E[p]\bigr)^G}
  \end{equation}
  yields that $H^1\bigr(G_{i+1},E[p]\bigr)=0$.
\end{proof}

\begin{lem}\label{h1_non_vanishing_lem}
    Let $p>2$ be a prime and $E/F$ an elliptic curve such that $G_2$ is greatest possible. If $G$ is among the exceptional cases in Theorem~\ref{thm1} or in Proposition~\ref{hom_prop} then $H^1\bigl(G_i,E[p^i]\bigr) \neq 0$ for all $i \geq 2$.
\end{lem}
\begin{proof}
We claim that the sequence~\eqref{inf_res_M_seq} is part of a short exact sequence. The next term in the sequence is $H^2\bigl(G,E[p]\bigr)$, and so it suffices to show that the map $H^1(M_{i+1},E[p])^G \to H^2\bigl(G,E[p]\bigr)$ is zero. By Lemma~\ref{h2_lem}, the target group is trivial unless $G = \sm{*}{*}{0}{1}$. If  $p > 3$ and $G = \sm{*}{*}{0}{1}$, then the source group $H^1(M_{i+1},E[p])^G$ vanishes. If $p=3$ and $G = \sm{*}{*}{0}{1}$, the source is cyclic and generated by a cocycle $\xi\colon M_{i+1} \to E[3]$ such that $\sm{a}{b}{c}{d} \mapsto \bigl(\begin{smallmatrix} c/3 \\ 0 \end{smallmatrix}\bigr)$. The image of $\xi$ in $H^2\bigl(G,E[p]\bigr)$ is zero because the formula $\sm{a}{b}{c}{d} \mapsto \bigl(\begin{smallmatrix} ac/3 \\ 0 \end{smallmatrix}\bigr)$ lifts it to a cocycle $G_{i+1}\to E[3]$.

Therefore, the dimension of $H^1(G_{i+1}, E[p])$ is the sum of the dimensions of the two groups surrounding it in the sequence~\eqref{inf_res_M_seq}. In all cases, this dimension is strictly larger than the dimension of the group of $p$-torsion points of $E$ defined over $F$.

  Now we turn to sequence~\eqref{mult_p_seq} with $i\geq 2$. In all cases, the dimension of $H^1\bigl(G_i,E[p]\bigr)$ is strictly larger than the dimension of $E(F)[p^{i-1}]/pE(F)[p^i]$ (which is at most $1$ because $E(F)[p^{i-1}]$ must be cyclic by the assumption on $F$). We conclude that $H^1\bigl(G_i,E[p^i]\bigr)$ is non-trivial.
\end{proof}

So far we have been able to treat all cases in which $G_i$ is greatest possible and $F \cap \QQ(\mu_p) = \QQ$. We will now restrict our attention to $F=\QQ$ and $p>3$. Luckily, for the large majority of elliptic curves over $\QQ$ the groups $G_i$ are indeed greatest possible. The following is a summary of the results in~\cite{greenberg} and in~\cite{grennbergetal}.

\begin{thm}\label{greenberg_thm}
  Let $p>3$ and $i>1$. Let $E/\QQ$ be an elliptic curve with an isogeny of degree $p$ defined over $\QQ$. Then $G_i$ is greatest possible except in two cases:
  \begin{itemize}
  \item when $p=7$ and the curve is the quadratic twist of a curve of conductor $49$, or
  \item when $p=5$ and there is an isogeny $\psi\colon E \to E''$ of degree $25$ defined over $\QQ$.
  \end{itemize}
\end{thm}

We will now treat the two exceptional cases, starting with $p=5$.
 \begin{lem}\label{25_lem}
  Let $E/\QQ$ be an elliptic curve and suppose there is an cyclic isogeny $\psi\colon E \to E'\to E''$ of degree $p^2 = 25$ defined over $\QQ$. Then $H^1\bigl(G_2,E[p^2]\bigr) = 0$ if and only if $H^1\bigl(G_i,E[p^i]\bigr) =0$ for all $i>1$. This vanishing holds except if $G=\sm{1}{*}{0}{*}$, if $G=\sm{*}{*}{0}{1}$, or if $E$ appears in Theorem~\ref{thm1} as an exception.
 \end{lem}
 For instance, it is non-vanishing if $E$ admits a rational $5$-torsion point or if $E'$ admits a rational $5$-torsion point. The curves 11a3 and 11a2 are examples of these two situations where $H^1\bigl(G,E[p]\bigr)=0$, yet $H^1\bigl(G_i,E[p^i]\bigr) \neq 0$ for all $i>1$ because there are two $5$-isogenies 11a3 $\to$ 11a1 $\to$ 11a2 with only 11a3 and 11a1 having a rational $5$-torsion point. The cohomology group $H^1\bigl(G_2,E[25]\bigr)$ is also non-trivial for 11a1 by Proposition~\ref{hi_vanishes_prop}.

 \begin{proof}
  Note that there are no elliptic curves with rational points of order $25$ and there are no cyclic isogenies over $\QQ$ of degree $p^3=125$. Greenberg shows in Theorem~2 in~\cite{greenberg} that the index of $G_2$ in $\GL_2\bigl(\ZZ/25\bigr)$ is divisible by $5$ but not $25$. Hence the group $G_2$ can be identified with a subgroup of the upper triangular matrices modulo $p^2$, but the top left entry is not constant $1$ modulo $p^2$ and the top right corner is not constant zero modulo $p$. Since the index is only divisible by $5$ once, the group $G_2$ consists of all the upper triangular matrices that reduce to an element of $G$.

  We wish to use the same strategy as in the proof of Proposition~\ref{hi_vanishes_prop}, but we have to show that $G$-fixed part of $H^1\bigl(M_2,E[p]\bigr)$ is still zero despite $M_2\neq M$. This time $M_2$ can be identified with upper triangular matrices modulo $p$ and the computations are slightly easier. One finds that $\Hom_G\bigl(M_2,E[\varphi]\bigr)$ has dimension $2$ if $G=\sm{1}{*}{0}{*}$ and $0$ in all other cases. Similarly, the dimension of $\Hom_G\bigl(M_2,E'[\hat\varphi]\bigr)$ is equal to $2$ if $G=\sm{*}{*}{0}{1}$ and zero otherwise. (Alternatively, it is not too hard to show by direct calculation that the dimension of $\Hom_G\bigl(M_2,E[p]\bigr)$ is $2$ if $G=\sm{1}{*}{0}{*}$, $1$ if $G=\sm{*}{*}{0}{1}$, and 0 otherwise.)

  Hence if we assume that neither $G = \sm{*}{*}{0}{1}$ nor $G=\sm{1}{*}{0}{*}$ nor $G= \Bigl\{\sm{v^2}{w}{0}{v}\Bigm\vert v\in\FF_p^{\times}, w \in \FF_p\Bigr\}$, then $H^1\bigl(G_i,E[p^i]\bigr)=0$ for all $i\geq 1$ with the same proof as in Proposition~\ref{hi_vanishes_prop}.

  If $G=\sm{1}{*}{0}{*}$, then one can show as in Lemma~\ref{h1_non_vanishing_lem} that $H^1\bigl(G_i,E[p^i]\bigr)$ is non-zero for all $i>1$. Similarly for $G = \Bigl\{\sm{v^2}{w}{0}{v}\Bigm\vert v\in\FF_p^{\times},w\in \FF_p\Bigr\}$.

  Finally if $G=\sm{*}{*}{0}{1}$, then one may compute $H^1\bigl(G_2,E[p^2]\bigr)$ directly: the group $G_2$ consists of all upper triangular matrices modulo $p^2$ whose lower right entry is congruent to $1$ modulo $p$. Let $H$ be the subgroup generated by $\sm{1}{1}{0}{1}$. Then the method used in the proof of Theorem~\ref{thm1} shows that the subgroup of $H^1\bigl(H,E[p^2]\bigr)$ fixed by the action of $G_2/H$ is trivial. However $H^1\bigl(G_2/H,E[p^2]^H\bigr)\cong \ZZ/p\ZZ$ implies then that $H^1\bigl(G_2,E[p^2]\bigr)\cong \ZZ/p\ZZ$ where an explicit isomorphism sends a cocycle $\xi$ to the first coordinate of $\xi\bigl(\sm{1}{0}{0}{1+p}\bigr)$ in $ p\,\ZZ/p^2\ZZ$. From the exact sequence~\eqref{mult_p_seq}, one deduces that $H^1\bigl(G_2,E[p]\bigr)$ is non-trivial and again this implies that all $H^1\bigl(G_i,E[p^i]\bigr)$ are non-zero for $i>1$.
 \end{proof}

 \begin{lem}\label{49_lem}
   Let $E/\QQ$ be a quadratic twist of a curve of conductor $49$ and let $p=7$. Then $H^1\bigl(G_i,E[p^i]\bigr) = 0$ for all $i\geq 1$.
 \end{lem}
 \begin{proof}
   Assume first that $E$ is one of the curves of conductor $49$. By assumption $E$ has complex multiplication by $\mathcal{O}$, where $\mathcal{O}$ is either $\ZZ[\sqrt{-7}]$ or the ring of integers in $\QQ(\sqrt{-7})$. Since $\QQ(\sqrt{-7})\subset K$, the subgroup $\Gal(K/\QQ(\sqrt{-7})) < G$ acts by $\mathcal{O}$-linear endomorphisms on $E[7^i]$. By scaling with the period, we may choose points $p$ and $\sqrt{-7} \cdot p$ as a basis for $E[7^i]$. Any lift of these forms a $\ZZ_7$-basis of the Tate module $T_7 E$. The endomorphism $a+b\sqrt{-7}$ with $a$, $b\in \QQ\cap \ZZ_7$ acts via $\sm{a}{b}{-7b}{a}$ on $T_7(E)$.

   The Frobenius element $\Fr_{\ell} \in \GL(T_7E)$ for $\ell = 347$ has trace $a_{\ell} = 4$ for all four curves of conductor $49$. Since $\ell$ splits in $\QQ(\sqrt{-7})$, the Frobenius $\Fr_{\ell}$ in $\GL_2(\ZZ_7)$ is a matrix of the above shape with trace $4$ and determinant $347$. We find that it is congruent to $\sm{2}{0}{0}{2}$ modulo $7$. Since $G$ contains now the homotheties by $2$ and $4$, the result follows from Lemma~\ref{homothety_lem}.

   Let now $E$ be a quadratic twist of a curve of conductor $49$. Then it is the quadratic twist of one of them by an integer $D$ coprime to $7$. The above homotheties are multiplied by a non-zero scalar and hence Lemma~\ref{homothety_lem} also implies the result for $E$.
 \end{proof}

 \begin{proof}[Proof of Theorem~\ref{thmi}]
   We combine the results from Proposition~\ref{hi_vanishes_prop}, Lemma~\ref{h1_non_vanishing_lem}, Lemma~\ref{25_lem} and Lemma~\ref{49_lem}. From these we conclude immediately that $H^1\bigl(G_i,E[p^i]\bigr)=0$ if and only if $H^1\bigl(G_2,E[p^2]\bigr)=0$.

   We are now left with making the list in Theorem~\ref{thmi} match with the non-vanishing cases. We start by verifying that the cohomology groups are non-vanishing in each of the five special cases in the theorem.

   \begin{itemize}
   \item First, if $E$ contains a rational point of order $p$, then $G=\sm{1}{0}{0}{*}$, $G=\sm{*}{0}{0}{1}$, or $G=\sm{1}{*}{0}{*}$.  In all these cases, the cohomology groups in question do not vanish by Lemma~\ref{h1_non_vanishing_lem} and Lemma~\ref{25_lem}.

   \item In the second point in the list of Theorem~\ref{thmi}, $p=5$ and the quadratic twist by $D=5$ of $E$ has a rational $5$-torsion point. Then $G$ is contained in $\bigl\{ \sm{v^2}{*}{0}{v}\bigm\vert v\in\FF_p^{\times}\bigr\}$. If $G$ is equal to that group, then Theorem~\ref{thm1} and Lemma~\ref{h1_non_vanishing_lem} or Lemma~\ref{25_lem} imply the non-vanishing of $H^1\bigl(G_2,E[p^2]\bigr)$. Otherwise we may choose the basis of $E[p]$ so that $G$ is contained in the diagonal matrices of this form, in which case Lemma~\ref{h1_non_vanishing_lem} proves the assertion.

   \item In the third point, $p=5$ and the quadratic twist by $D=5$ of $E'$ has a rational $5$-torsion point. Then $G$ is contained in $\bigl\{ \sm{u}{*}{0}{u^2}\bigm\vert u\in\FF_p^{\times}\bigr\}$. There is no isogeny of degree $25$ defined over $\QQ$ leaving from $E$, hence $G_2$ is greatest possible; therefore Lemma~\ref{h1_non_vanishing_lem} proves the desired non-vanishing.

   \item If we are in the situation of the fourth point in Theorem~\ref{thmi}, we are in the situation of Lemma~\ref{25_lem} and $G=\sm{*}{*}{0}{1}$. Therefore $H^1\bigl(G_2,E[p^2]\bigr)\neq 0$.

   \item In the final point, if $p=11$ and $E$ is 121c2, then Theorem~\ref{thm1} and Lemma~\ref{h1_non_vanishing_lem} shows the desired non-vanishing. If the curve is 121c1 instead, then $G=\bigl\{ \sm{u}{*}{0}{u^2}\bigm\vert u\in\FF_p^{\times}\bigr\}$ and Lemma~\ref{h1_non_vanishing_lem} treats this case too.
   \end{itemize}

   Next, we have to check that every case when the group $H^1\bigl(G_2,E[p^2]\bigr)$ is non-trivial is among the exceptional cases of Theorem~\ref{thmi} above.

   Let us assume first that $G_2$ is greatest possible and consider the cases in Lemma~\ref{h1_non_vanishing_lem}. If $G=\sm{1}{0}{0}{*}$ or $G=\sm{1}{*}{0}{*}$, then $E$ has a rational $p$-torsion point. By Mazur's Theorem on the torsion point on elliptic curves over $\QQ$, we know that this can only occur if $p=5$ or $p=7$ and we fall under the first point in the list of Theorem~\ref{thmi}. If $(E,p)$ appears as an exception in Theorem~\ref{thm1}, then either $p=5$ and we are in the situation of the second point, or $p=11$ and we are in the last point on the list. If $G$ is the group of all matrices of the form $\sm{v^2}{0}{0}{v}$, then the quadratic twist by $D=5$ has a rational $5$-torsion point and we are in the situation of the second point. Finally, assume $G$ is contained in the group $\bigl\{ \sm{u}{*}{0}{u^2}\bigm\vert u\in\FF_p^{\times}\bigr\}$. Then $p\equiv 2\pmod{3}$. If $p=5$, then the quadratic twist by $D=5$ of $E'$ has a rational $5$-torsion point and we are in the third case. If $p\geq 11$, then the proof that there is only one curve, namely 121c1, is very analogous to Lemma~\ref{121c2_lem}.

   Finally, we consider the cases in Lemma~\ref{25_lem}. If $G=\sm{1}{*}{0}{*}$, then $E$ has a rational $5$-torsion point and we are in the first point in the list. If $G=\sm{*}{*}{0}{1}$, then $E'$ admits a rational $5$-torsion point, which is the fourth point on the list. If $(E,p)$ appear as exceptions in Theorem~\ref{thm1}, then we fall into the second point on the list.
  \end{proof}

\section{Numerical computations}\label{sec:numer-comp}

 We used Magma~\cite{magma} to perform, for small primes $p$, the numerical computation of our cohomology group $H^1\bigl(G_2,V_2\bigr)$ for various subgroup $G_2 \leq \GL_2(\ZZ/p^2)$, where $V_2$ is the natural rank $2$ module over $\ZZ/p^2$ on which $G_2$ acts. We restricted our attention to groups with surjective determinants and we only considered groups up to conjugation in $\GL_2(\ZZ/p^2)$.

 We will continue to write $M_2$ for the kernel of reduction $G_2\to \GL_2(\FF_p)$ and $G$ for its image.

\subsection{$p=2$}
 For the prime $p=2$, the groups $H^1\bigl(G_2,V_2\bigr)$ are non-zero for $36$ conjugacy classes of subgroup $G_2\leq \GL_2(\ZZ/4)$ with surjective determinant. The possible cohomology groups are $\bigl(\ZZ/2\bigr)^k$ for $0\leq k\leq 6$ and $\ZZ/4$.  Non-trivial cohomology groups appear for all dimensions $1\leq d \leq 4$ of $M_2$.

\subsection{$p=3$}
 There are $41$ groups $G_2$ with non-vanishing $H^1\bigl(G_2,V_2\bigr)$. For thirteen of them the cohomology group is $\ZZ/3\oplus \ZZ/3$, for one it is $\ZZ/3\oplus\ZZ/3\oplus \ZZ/3$ and for all others it is just $\ZZ/3$. In all non-vanishing cases the image $G\leq \GL_2(\FF_3)$ of reduction has either non-trivial $H^0(G,V)$ or non-trivial $H^2(G,V)$, where $V$ is the $2$-dimensional vector space over $\FF_3$ with its natural action by $G\leq \GL_2(\FF_3)$. In other words, these numerical computations show that if the group $H^1\bigl(G_2,E[9]\bigr)$ is non-trivial for an elliptic curve $E/\QQ$, there is an isogeny $\varphi\colon E \to E'$ defined over $\QQ$ of degree $3$ such that either $\varphi$ or its dual $\hat\varphi$ has a rational $3$-torsion point in its kernel.

 The maximal order of the cohomology group appears for the group $G_2$ consisting of all matrices in $\GL_2(\ZZ/9)$ with reduction $\sm{1}{0}{0}{1}$ or $\sm{1}{0}{0}{-1}$ modulo $3$.

\subsection{$p=5$}
 There are $39$ groups $G_2$ with non-vanishing $H^1\bigl(G_2,V_2\bigr)$. For two of them, the group is $\ZZ/5\oplus\ZZ/5$, for one it is $\ZZ/25$ and for all others it is $\ZZ/5$. If we restrict to those groups for which $M_2$ has dimension $4$, then there are five cases as found in Section~\ref{results_i_sec}:
 \begin{center}
  \begin{tabular}{c|c|c|c|c|c}
   $G$                      & $\sm{v^2}{0}{0}{v}$ & $\sm{1}{0}{0}{*}$ & $\sm{u}{*}{0}{u^2}$ & $\sm{v^2}{*}{0}{v}$ & $\sm{1}{*}{0}{*}$ \\
   $|G|$                    & $4$                 & 4                 & 20                  & 20                                               & 20 \\
   $H^1\bigl(G_2,V_2\bigr)$ & $\ZZ/5$             & $\ZZ/5\oplus\ZZ/5$ & $\ZZ/5$ & $\ZZ/5$ & $\ZZ/5$
  \end{tabular}
 \end{center}
 This determines what the non-vanishing cohomology groups can be for this specific prime.

\subsection{$p=7$}
 Here we restricted our attention to the subgroups $G_2$ for which $M_2$ has dimension $4$. Then, as previously found, there are only two cases. The group $G$ can be of the form $\sm{1}{0}{0}{*}$ or $\sm{1}{*}{0}{*}$. In the first case the cohomology group $H^1\bigl(G_2,V_2\bigr)$ is $\ZZ/7\oplus\ZZ/7$; in the latter it is $\ZZ/7$.

\section{Applications to local and global divisibility of rational points}\label{sec:divisibility}

The cohomology groups that we have discussed in this paper also appear in the analogue of the Grunwald--Wang problem for elliptic curves. This question was raised by Dvornicich and Zannier in~\cite{dvornicich_zannier}.
\begin{gwqn}
 Let $E/\QQ$ be an elliptic curve, $P\in E(\QQ)$, and $m>1$. If $P$ is divisible by $m$ in $E(\QQ_\ell)$ for almost all $\ell$, is it true that $P$ is divisible by $m$ in $E(\QQ)$ ?
\end{gwqn}

By the Chinese remainder theorem, it is sufficient to restrict to the case when $m=p^i$ is a prime power. The answer is positive if $m$ is prime. The explicit example in~\cite{dvornicich_zannier_4} shows that the answer is negative for $m=4$. In~\cite{paladino_5}, it is shown that the answer is positive for all $m=p^2$ with $p$ a prime larger than $3$. To our knowledge, the case $m=9$ has not been determined.

This question connects to our cohomology groups through the following reinterpretation. Suppose $m=p^i$ for our fixed prime $p$. Let $\Sigma$ be a finite set of places in $\QQ$. Let
\begin{equation*}
  D(E/\QQ) = \ker \Bigl ( E(\QQ)/p^i E(\QQ) \to \prod_{v\not\in \Sigma} E(\QQ_v)/p^i E(\QQ_v) \Bigr)
\end{equation*}
be the group that measures if there are points $P$ that are locally divisible by $p^i$, but not globally. Let
\begin{equation}\label{ln_eq}
  L(E/\QQ) = L(G_i)=\ker \Bigl ( H^1\bigl( G_i, E[p^i] \bigr) \to \prod_{\substack{C \leq G_i\\C\text{ cyclic}}} H^1\bigl( C, E[p^i]\bigr) \Bigr)
\end{equation}
be the kernel of reduction to all the cyclic subgroups of $G_i$. We now assume that $\Sigma$ contains all places above $p$ and all bad places. By Chebotarev's theorem, $L(E/\QQ)$ is also the kernel of localization from $H^1\bigl( G_i, E[p^i] \bigr)$ to all $H^1(D_{w\mid v}, E[p^i]\bigr)$ where $D_{w\mid v}$ is the decomposition group in $K_i/\QQ$ of a place $w$ above $v$. Hence a natural notation for $L(E/\QQ)$ could be $\Sha^1\bigl(U, E[p^i]\bigr)$ with $U$ the complement of $\Sigma$ in $\operatorname{Spec}(\ZZ)$. The sequence
\begin{equation*}
  \xymatrix@1{ 0\ar[r] & D(E/\QQ) \ar[r] & L(E/\QQ)\ar[r] & H^1\bigl(G_i, E(K_i)\bigr) }
\end{equation*}
is exact. Hence the answer is positive for $m=p^i$ if $H^1\bigl(G_i, E[p^i]\bigr)$ vanishes. Note that the description of $L(E/\QQ)$ in~\eqref{ln_eq} is now entirely group-theoretic, and can be computed numerically with the methods described in the previous section.

\begin{thm}\label{div_thm}
  Let $p$ a prime and $i\geq 1$. Then the Grunwald--Wang problem for local-global divisibility by $m=p^i$ admits a positive answer for all elliptic curves $E/\QQ$ if and only if $p>3$ or $m=2$ or $m=3$.
\end{thm}
\begin{proof}
  If we find a point $P$ of infinite order that is a counter-example for $m=p^i$, then $p^jP$ is a counter-example for $m=p^{i+j}$ for any $j>0$. As mentioned before, the negative answer for $m=4$ is explained in~\cite{dvornicich_zannier_4}. This settles also all higher powers of $2$ as their examples are points of infinite order. Counter-examples when $m$ is a power of $3$ were first found by Creutz in~\cite{creutz}. We will give below in Proposition~\ref{div9_prop} a new counter-example of infinite order for $m=9$. For $p\geq 5$ the theorem follows from~\cite{paladino_5}. However, we wish to give a slightly simplified proof with our methods.

 Assume therefore $p\geq 5$. We will now show that the kernel of localization $L(G_i)$ is zero.  Note that by Greenberg's result in Theorem~\ref{greenberg_thm} and the work done in the exceptional cases in Lemma~\ref{25_lem} and Lemma~\ref{49_lem}, we may assume that $G_i$ is greatest possible or that $i=2$ and $M_2$ consists of all matrices $m$ such that $m-1$ is upper triangular. In both cases the elements in $E[p^i]$ fixed by $M_i$ are just $E[p]$. We get an exact sequence
 \begin{equation*}
  \xymatrix@1{ 0\ar[r] & H^1\bigl(G,E[p]\bigr) \ar[r]^{\inf} & H^1\bigl(G_i,E[p^i]\bigr) \ar[r] &  H^1\bigl( M_i, E[p^i]\bigr). }
 \end{equation*}
 First assume that $L(G_i)$ contains a non-trivial element which belongs to the image of the inflation map from $H^1\bigl( G, E[p]\bigr)$. Since the latter must now be non-trivial, $G$ must contain the element $\bar h =\sm{1}{1}{0}{1}$. By the description of $M_i$, we find that $G_i$ contains the element $h= \sm{1}{1}{0}{1}$. Let $C$ be the cyclic group generated by $h$ and let $\bar C$ be its image in $G$. Our computations for proving Theorem~\ref{thm1} showed that $H^1\big(G,E[p]\bigr)\to H^1(\bar C, E[p]\big)$ is a bijection. Next, both maps in the composition
\begin{equation*}
  \xymatrix@1{H^1\bigl( \bar C, E[p] \bigr) \ar[r] & H^1\bigl( \bar C, E[p^i]^{C\cap M_i}\bigr)\ar[r] & H^1\bigl(C, E[p^i]\bigr)
 }
\end{equation*}
 are injective: for the latter it is because any inflation map is injective, and for the first it can be read off the long exact sequence associated to the inclusion $E[p]\to E[p^i]^{C\cap M_i}$. We conclude that $H^1\bigl(G,E[p]\bigr)\to H^1\bigl(C,E[p^i]\bigr)$ is injective. This now contradicts the assumption that $L(G_i)$ contained a non-trivial element from $H^1\bigl(G,E[p]\bigr)$.

 Therefore, $L(G_i)$ injects into to
 \begin{equation*}
 L(M_i) = \ker \Bigl( H^1\bigl( M_i, E[p^i] \bigr)\to \prod_{\substack{C\leq M_i\\ \text{cyclic}}} H^1\bigl(C,E[p^i]\bigr) \Bigr),
 \end{equation*}
 where the product now runs over all cyclic subgroups of $M_i$. We will now prove by induction on $i$ that $L(M_i)$ is trivial. It is known for $i=1$.

 Recall that the group $H_i$ acts trivially on $E[p^i]$. We consider the following diagram with exact rows:
 \begin{equation}\label{liind_eq}
   \xymatrix{
     0\ar[r] & H^1\bigl(M_i,E[p^i]\bigr)\ar[r] \ar[d]& H^1\bigl(M_{i+1},E[p^i]\bigr)\ar[r]\ar[d] & H^1\bigl(H_i,E[p^i]\bigr) \ar[d]\\
     0\ar[r] & \prod_C H^1\bigl(C/C\cap H_i, E[p^i]^{C\cap H_i}\bigr) \ar[r] &
      \prod_C H^1\bigl(C, E[p^i]\bigr) \ar[r] &
      \prod_C H^1\bigl(C\cap H_i, E[p^i]\bigr)
   }
 \end{equation}
 where the products run over all cyclic subgroups $C$ of $M_{i+1}$. Now the vertical map on the right hand side has the same kernel as
 \begin{equation*}
   \xymatrix@1{H^1\bigl(H_i,E[p^i]\bigr) = \Hom\bigl(H_i,E[p^i]\bigr)\ar[r] & \prod\limits_{\substack{D\leq H_i\\ \text{cyclic}}} \Hom\bigl( D, E[p^i]\bigr) }
 \end{equation*}
 and this map is clearly injective. Since $C\cap H_i$ fixes $E[p^i]$ the vertical map on the right in the above diagram~\eqref{liind_eq} is injective by induction hypothesis because $C/C\cap H_i \cong CH_i/H_i$ will run through all cyclic subgroups of $M_i$ at least once. Therefore the middle vertical map in~\eqref{liind_eq} is injective, too.

 Next, consider the following diagram with exact rows:
 \begin{equation*}
  \xymatrix{ 0\ar[r] & E[p] \ar[r]^(0.3){\delta} & \Hom\bigl(M_{i+1},E[p]\bigr) \ar[r]^{\iota}\ar[d] & H^1\bigl(M_{i+1}, E[p^{i+1}]\bigr) \ar[r]^{[p]} \ar[d] & H^1\bigl(M_{i+1}, E[p^i]\bigr)\ar[d] .\\
         E[p^{i+1}]^C\ar^{[p]}[r]    & E[p^i]^{C} \ar[r]^(0.3){\delta_C} & \Hom\bigl(C,E[p]\bigr) \ar[r] & H^1\bigl(C, E[p^{i+1}]\bigr) \ar[r] & H^1\bigl(C, E[p^i]\bigr) }
 \end{equation*}
 Here $C$ is any cyclic subgroup of $M_{i+1}$. The zero at the top left corner is a consequence from the fact that the $M_{i+1}$-fixed points in $E[p^j]$ are exactly $E[p]$ for all $1\leq j \leq i+1$.

 If $\xi\in L(M_{i+1})$, then its image under $[p]$ in $H^1\bigl(M_{i+1}, E[p^i]\bigr)$ must be trivial by what we have shown for the middle vertical map in~\eqref{liind_eq}. Therefore $\xi$ is the image under $\iota$ of an element $f$ in $\Hom\bigl(M_{i+1}, E[p]\bigr)$. Since $E[p]$ is $p$-torsion, we can identify $\Hom\bigl(M_{i+1}, E[p]\bigr)$ with $\Hom\bigl(M_2,E[p]\bigr)$. To say that $\xi$ restricts to zero for a cyclic group $C\leq M_{i+1}$ forces $f\colon M_2\to E[p]$ to be in the image of the map $\delta_C\colon E[p]\to\Hom\bigl(C,E[p]\bigr)$ for all cyclic subgroups $C$ of $M_2$.

 Now we identify $M_2$ with the additive subgroup $\tilde M_2\leq \Mat_{2}(\FF_p)$ as before. Under this identification the map $\delta$ sends a $p$-torsion point $T\in E[p]=\FF_p^2$ to the map $f$ sending a matrix $m\in \tilde M_2$ to $m(T)$. Thus, the restriction of $f$ to $\Hom\bigl(\langle m \rangle, \FF_p^2\bigr)$ is in the image of $\delta_{\langle m \rangle}$ for a particular $m\in\tilde M_2$ if and only if $f(m)\in\FF_p^2$ belongs to the image of $m$. Therefore, we have shown that
 \begin{equation} \label{ln_hom_eq}
  L(M_{i+1}) = \frac{ \Bigl\{f \in \Hom\bigl(\tilde M_2, \FF_p^2\bigr) \ \Bigm\vert f(m) \in \im(m)\  \forall m\in \tilde M_2 \Bigr\} }{ \Bigl\{ f(m) = m(T) \text{ for some }T\in \FF_p^2\Bigr\} }.
 \end{equation}
 We wish to show that $L(M_{i+1})$ is trivial if $\tilde M_2$ is the full matrix group or the upper triangular matrices. Assume first that $\tilde M_2$ is the full matrix group. Then $f$ is determined by its image on the matrices with only one non-zero entry. However, the local condition of being in the image of $\delta_C$ for these matrices and the matrices $\sm{1}{0}{1}{0}$ and $\sm{0}{1}{0}{1}$ forces $f(m)$ to be just $m(T)$ for $T= f\bigl( \sm{1}{0}{0}{0} \bigr) + f\bigl(\sm{0}{0}{0}{1} \bigr)$. Therefore $L(M_{i+1})$ is trivial. The case when $\tilde M_2$ is the group of upper triangular matrices is very similar.
\end{proof}

The result about the vanishing of $L(G_2)$ for $p>3$ in the above proof is reminiscent of Proposition~3.2.ii in~\cite{dvornicich_zannier}. We have reproved part of this result with a more conceptual approach. The main reason for doing so is that the general statement there is slightly incorrect. The case $\dim(M_2) = 3$ assumes that $\sm{1}{1}{0}{1}$ belongs to $G_2$. However, for $p=3$, the group $G_2$ generated by $\sm{7}{8}{3}{1}$ and the group of all matrices $m$ with $m-1$ upper-triangular is a counterexample. This group does not contain any elements of order $9$ and one can compute that $L(G_2)$ is isomorphic to $\ZZ/3$.

%\todo[inline]{The description in~\eqref{ln_hom_eq} seems to be of independent interest, but already implicit in~\cite{dvornicich_zannier}. This quotient can be nontrivial even when $\dim (M_2) = 3$: Take the group of trace zero matrices and $p=2$.}
We include here a new counter-example for $m=9$; the method is quite different from~\cite{creutz} where a first such example was found.
\begin{prop}\label{div9_prop}
  Let $E$ be the elliptic curve labeled 243a2, given by the global minimal equation $y^2 + y = x^3+20$, and let $P=(-2,3)$. Then $3\, P$ is divisible by $9$ in $E(\QQ_{\ell})$ for all primes $\ell\neq 3$, but it is not divisible by $9$ in $E(\QQ)$.
\end{prop}
\begin{proof}
  Since $P$ is a generator of the free part of this curve of rank $1$, it is clear that $3P$ is not divisible by $9$ in $E(\QQ)$.

  Let $k$ be the unique subfield of $\QQ(\mu_9)$ of degree $3$ over $\QQ$ and let $\zeta$ be a primitive $9$-th root of unity. Then $P' = (3\zeta^5 + 3\zeta^4+3, 9\zeta^4-9\zeta^2+9\zeta+4)\in E(k)$ satisfies $3P'=P$. Thus, if $\ell\equiv \pm 1\pmod{9}$, then $\ell$ splits in $k$ and hence $P$ is divisible by $3$ in $E(\QQ_{\ell})$. As a consequence, $3P$ is divisible by $9$ over $\QQ_{\ell}$.

  For this curve and $p=3$, the group $G=\sm{1}{*}{0}{*}$ is of order $6$ and $K=\QQ(\theta)$ with $\theta^6+3=0$. Factoring the $9$-division polynomial, one finds that $K_2/K$ is an extension of degree $3$. The field of definition of the points of order $9$ is $K_2$ except for those points $T$ with $3\, T \in E(\QQ)[3]$. Those are defined instead over a non-Galois extension of degree $3$ over $k$.

  Let $\ell\not\equiv \pm 1 \pmod{9}$ and $\ell\neq 3$. Then the Frobenius element $\Fr_{\ell}$ in $G_2$ cannot belong to $\Gal(K_2/k)$. Therefore $\Fr_{\ell}$ does not fix any point of order $9$. It follows that $\tilde E(\FF_{\ell})[9] = \tilde E(\FF_{\ell})[3]$. Consider the following commutative diagram, whose lower row is exact.
  \begin{equation*}
    \xymatrix{
    &&P\in E(\QQ)/3 E(\QQ) \ar[r]^{[3]} \ar[d] & E(\QQ)/9 E(\QQ) \ar[d] \\
    E(\QQ_{\ell})[9] \ar[r]^{[3]} & E(\QQ_{\ell})[3] \ar[r]^(0.4){\delta} & E(\QQ_{\ell})/3E(\QQ_{\ell}) \ar[r]^{[3]} & E(\QQ_{\ell})/9E(\QQ_{\ell})
    }
  \end{equation*}
  Since $\ell\neq 3$, the reduction of $E$ at $\ell$ is good and hence $[3]$ is an isomorphism on the kernel of reduction $E(\QQ_{\ell})\to \tilde E(\FF_\ell)$. It follows that $E(\QQ_{\ell})[3]\cong \tilde E(\FF_{\ell})[3]$ and $E(\QQ_{\ell})/3E(\QQ_{\ell}) \cong \tilde E(\FF_{\ell})/3 \tilde E(\FF_{\ell})$ have the same size. By the above argument $\delta$ is an injective map between two groups of the same size. Thus $\delta$ is a bijection. This implies that $3P$ is divisible by $9$ in $E(\QQ_{\ell})$.
\end{proof}

This is the counter-example of smallest conductor for $m=9$; here is how we found that this curve is a likely candidate.

Consider curves $E$ with a $3$-isogeny where either the kernel has a rational $3$-torsion point or where the kernel of the dual isogeny has a rational $3$-torsion point. On the one hand, we computed (for a few thousand primes $\ell\neq 3$ of good reduction) the pairs $(a_{\ell}(E), \ell)$ modulo $9$. On the other hand, we may determine all subgroups $G_2\leq \GL_2(\ZZ/9)$ with surjective determinant to find the examples for which the kernel~\eqref{ln_eq} is non-trivial. There are 13 such groups. The dimension of $M_2$ in these cases is $1$, $2$ or $3$. For each of them we may list pairs $(\operatorname{tr}(g), \det(g))$ when $g$ runs through all matrices $g \in G_2$.

Now, if the list of possible pairs $(a_{\ell}(E), \ell)$ modulo $9$ agrees with one of the lists above, then $G_2$ could be among the groups for which the localization kernel is non-trivial. Furthermore, it is easy to check local divisibility for primes $\ell<1000$ for all possible candidates in $3E(\QQ)/9E(\QQ)$. The above curve 243a2 was the first to pass all these tests.

Here are a few more candidates. Note that we have not formally proved that local divisibility holds by $9$ holds for all primes $\ell$ of good reduction.

The point $P=(6,17)$ on the curve 9747f1 gives a point $3P$ which is likely to be locally divisible by $9$ for \textit{all} primes, but not divisible by $9$ globally. In this example $G_2$ has $54$ elements.

On the curve 972d2 the point $3P$ with $P=(13,35)$ is likely to be locally divisible by $9$ for all places $\ell\neq 3$, yet not globally so. This is a curve without a rational $3$-torsion point and $G_2$ having $54$ elements again.

All the above examples have complex multiplication by the maximal order in $\QQ(\sqrt{-3})$. The curve 722a1, with a point $P$ having $x$-coordinate $\tfrac{27444}{169}$, is an example without complex multiplication and $\vert G_2\vert =162$. Again, it is likely that $3P$ is locally divisible by $9$ at all places $\ell\neq 19$, but $3P$ is not globally divisible by $9$. The group $G_2$ here is probably conjugate to the one mentioned earlier as a counter-example to Proposition~3.2.ii in~\cite{dvornicich_zannier}.

We have also done numerical calculation of the kernel in~\eqref{ln_eq} for other primes. For $p=5$, there are only three subgroups $G_2$ in $\GL_2(\ZZ/25)$ with non-trivial localization kernel. They all have $\dim (M_2) = 2$ and $\vert G\vert =4$.

For $p=2$, there are twelve cases. The dimensions of $M_2$ can be $1$, $2$, or $3$. In only one of these cases is the localization kernel is $\ZZ/2\oplus\ZZ/2$; otherwise it is $\ZZ/2$.

%=========================================

\bibliographystyle{amsplain}
\bibliography{h1div}

\end{document}